\documentclass[11pt,leqno]{amsproc}
\usepackage[margin=1in]{geometry}
\usepackage{amsmath, amsthm, amssymb}
\usepackage[colorlinks=true, pdfstartview=FitV, linkcolor=blue,citecolor=blue, urlcolor=blue]{hyperref}
\usepackage[abbrev,lite,nobysame]{amsrefs}
\usepackage{times}
\usepackage[usenames,dvipsnames]{color}
\usepackage[T1]{fontenc}
\usepackage[utf8]{inputenc} 

\definecolor{labelkey}{rgb}{0,0,1}
\usepackage{dsfont}
\newenvironment{altproof}[1]
{\noindent
	{\em Proof of {#1}}.}
{\nopagebreak\mbox{}\hfill $\Box$\par\addvspace{0.5cm}}

\def\R {\mathbb{R}}

\newcommand{\vp}{\varphi}

\newcommand{\eps}{\varepsilon}

\def\supp{\mathrm{supp}}

\def\Z{\mathbb{Z}}
\def\d{\diamond}

\def\N{\mathbb{N}}
\def\R{\mathbb{R}}

\def\D{\mathcal{D}}

\newcommand{\cC}{{\mathcal C}}
\newcommand{\cD}{{\mathcal D}}

\newcommand{\cM}{{\mathcal M}}

\newcommand{\cP}{{\mathcal P}}

\newcommand{\Om}{\Omega}

\newcommand{\weakto}{\rightharpoonup}

\newcommand{\tu}{\widetilde{u}}

\def\pr{\right )}
\def\le{\left (}
\def\gg{^{\ast\ast}}
\def\d{\,d}
\def\D{\mathcal{D}^{2,2}(\R^N)}
\def\f{\varphi}
\def\supp{\mathrm{supp}\,}
\def\e{\varepsilon}

\newtheorem{proposition}{Proposition}[section]
\newtheorem{theorem}[proposition]{Theorem}

\newtheorem{lemma}[proposition]{Lemma}
\theoremstyle{definition}

\numberwithin{equation}{section}

\title[Biharmonic nonlinear scalar field equations]{Biharmonic nonlinear scalar field equations}

\author[J. Mederski]{Jaros\l aw Mederski}
\address[J. Mederski]{\newline\indent
	Institute of Mathematics,
	\newline\indent 
	Polish Academy of Sciences,
	\newline\indent 
	ul. \'Sniadeckich 8, 00-656 Warsaw, Poland
	\newline\indent 
	and
	\newline\indent 
	Department of Mathematics,
	\newline\indent 
	Karlsruhe Institute of Technology (KIT), 
	\newline\indent 
	D-76128 Karlsruhe, Germany
}
\email{\href{mailto:jmederski@impan.pl}{jmederski@impan.pl}}

\author[J. Siemianowski]{Jakub Siemianowski}
\address[J. Siemianowski]{\newline\indent
	Faculty of Mathematics and Computer Sciences,
	\newline\indent 
	Nicolaus Copernicus University in Toru\'{n} 
	\newline\indent 
	ul. Gagarina 11, 87-100 Toruń, Poland
	\newline\indent 
	and
	\newline\indent 
	Institute of Mathematics,
	\newline\indent 
	Polish Academy of Sciences,
	\newline\indent 
	ul. \'Sniadeckich 8, 00-656 Warsaw, Poland
	\newline\indent 
}
\email{\href{mailto:jsiem@mat.umk.pl}{jsiem@mat.umk.pl}}



\begin{document}

\begin{abstract}
We prove a Brezis-Kato-type regularity result for weak solutions to the biharmonic nonlinear equation
$$	\Delta^2 u = g(x,u)\qquad\text{in }\R^N$$
with a Carath\'eodory function $g:\mathbb{R}^N\times \mathbb{R}\to \mathbb{R}$, $N\geq 5$. The regularity results give rise to the existence of ground state solutions provided that $g$ has a general subcritical growth at infinity. We also conceive a new biharmonic logarithmic Sobolev inequality
$$
\int_{\mathbb{R}^N}|u|^2\log |u|\,dx\leq\frac{N}{8}\log \le C\int_{\mathbb{R}^N}|\Delta u|^2\,dx \pr, \quad\text{for } u \in H^2(\mathbb{R}^N), \; \int_{\mathbb{R}^N}u^2\d x = 1,
$$
for a constant $0<C< \Big(\frac{2}{\pi e N}\Big)^2$ and we characterize its minimizers.
\end{abstract}

\maketitle

\section{Introduction}

The study of higher-order differential elliptic operators is important, e.g. in nonlinear elasticity  \cite{Antman},  low Reynolds number hydrodynamics, in
structural engineering \cite{Selvadurai,Meleshko} as well as in nonlinear optics \cite{Fibich},
and has attracted attention from the mathematical point of view \cite{Gazzola}.  The methods developed for the second order problem, e.g. involving the Laplacian $-\Delta$, may no longer be available. For instance, it is the well-known that the bi-Laplacian $(-\Delta)^2=\Delta^2$ cannot be studied by means of some classical methods such as maximum principles, Polya-Szeg\H{o} inequalities, or even if $(\Delta u)^2\in L^1(\R^N)$, then it is possible that  $\Delta |u|\notin L^1_{loc}(\R^N)$. 

The first aim is of this work is to establish a regularity result in the spirit of Brezis-Kato \cite{BrezisKato} 
of weak solutions to 
\begin{equation}\label{eq}
	\Delta ^2 u = g(x,u),\quad x\in\Om,
\end{equation}
where $\Om\subset \R^N$ is a domain, $N\geq 2$ and $g:\Om\times\R\to\R$ is a Carath\'eodory function. If we suppose that  $\Om$ is bounded, then there is an extensive literature devoted to this problem. Namely, recall that if $g(x,u)=f(x)$, then Agmon, Douglis, Nirenberg \cite{Agmon_et_al} showed that for  $1<q<\infty$, $f\in L^q(\Om)$, there exists a unique strong solution $u\in W^{2,2}_0(\Om)\cap W^{4,q}(\Om)$ to \eqref{eq} provided that $\partial\Om\in C^4$ see also \cite[Corollary 2.21]{Gazzola} and references therein. Recently Mayboroda and Maz'ya \cite{Mayboroda_Mazja} showed $L^\infty$-estimates of $u$ (resp. $\nabla u$), where $f\in C_0^{\infty}(\Om)$, $\Om$ is an arbitrary bounded domain and  $N=4,5$
(resp. $N=2,3$). To the best of our knowledge, a variant of Brezis-Kato result \cite{BrezisKato} for \eqref{eq} is known only on a bounded domain in a particular case. Namely, Van der Vorst \cite{Van_der_Vorst} showed that, if $N\geq 5$, $g(x,u)=a(x)u$ and $a(x)\in L^{N/4}(\Om)$, then any weak solution $u\in W^{1,2}_0(\Om)\cap W^{2,2}(\Om)$ to \eqref{eq} satisfies $u\in L^q(\Om)$ for all $1\leq q<\infty$. This result is  suitable to show the regularity for the biharmonic equation with the nonlinearities of the special form $g(x,u)=f(u)u$ cf. \cite[Lemma B3]{Van_der_Vorst}.  In this paper we give a full answer to the problem on an arbitrary domain and for general $g$ with the adequate Brezis-Kato growth as we shall see below.

From now on we assume that $\Om\subset \R^N$ possibly unbounded domain and $N\geq 5$. Inspired by \cite{BrezisKato}, we impose on $g$ the 
following growth assumption:
\begin{equation}\label{eq:BK}
|g(x,s)|\leq a(x) \big(1+|s|\big),\quad\hbox{for }s\in\R\hbox{ and a.e. } x\in\Om,\hbox{ where } 0\leq a\in L^{N/4}_\text{loc}(\Om).
\end{equation}

The first main result reads as follows.
\begin{theorem}\label{th:BK}
	Let  $u\in W^{2,2}_\text{loc}(\Om)$ be a weak solution to \eqref{eq}, where $g$ satisfies \eqref{eq:BK}.
	Then $u\in C^{3,\alpha}_\text{loc}(\Om)\cap W^{4,q}_\text{loc}(\Om)$, for any $0<\alpha < 1$ and $1\leq q<\infty$.
\end{theorem}

It is worth mentioning that in proof of Theorem \ref{th:BK} we can no longer apply classical techniques for Laplacian, e.g. due to Brezis and Kato  \cite{BrezisKato}, or Brezis and Lieb \cite[Theorem 2.3]{BrezisLieb}, since $\Delta |u|$ may not be well-defined for $u\in W^{2,2}_\text{loc}(\Om)$. Moreover, the Moser iteration technique does not seem to be applicable straightforwardly for $g$.

We shall present some consequences of Theorem \ref{th:BK} in $\Om=\R^N$.
Let us define $\mathcal{D}^{2,2}(\R^N)$ as a completion of the space $C^\infty_0(\R^N)$ with respect to the norm
$
\|u\|_{\cD^{2,2}} := \le \sum_{|\alpha|=2}\|\partial^\alpha u\|_{L^2(\R^N)}^2 \pr^\frac{1}{2}$. By the use of the Fourier transform and the Plancharel theorem we find a constant $c>0$ such that, for $u\in C^\infty_0(\R^N)$, 
\[
\frac{1}{c}\|u\|_{\D} \leq \|\Delta u\|_{L^2(\R^N)} \leq c \|u\|_{\D}.
\]
Therefore, the norms $\|u\| := \|\Delta u\|_{L^2(\R^N)}$ and $\|u\|_{\D}$ are equivalent on $\D$.
Moreover, $\D$ is a Hilbert space with the inner product
\[
\langle u,v \rangle := \int_{\R^N} \Delta u \Delta v \d x \qquad \hbox{for }u,\; v\in \D
\]
and $u\in\D$ is a {\em weak solution} to \eqref{eq} provided that
$$\langle u,v \rangle=\int_{\R^N}g(x,u)v\qquad\hbox{for any }v\in C_0^{\infty}(\R^N).$$

As usually expected, the following general Poho\v{z}aev-type result holds, cf. \cite{PucciSerrin}.

\begin{theorem}\label{th:Poho} Let $u\in\D$ be a weak solution to \eqref{eq}, where $g$ satisfies \eqref{eq:BK}.
	Then 
	\begin{equation}\label{Poh_identity}
		\int_{\R^N}|\Delta u|^2\, dx =\frac{2N}{N-4}\int_{\R^N}G(x,u)\,dx 
		+\frac{2}{N-4}\int_{\R^N} x\cdot \partial_x G(x,u)\,dx.
	\end{equation}
	provided that $G(x,u)$, $x\cdot \partial_x G(x,u)\in L^1(\R^N)$, where $G(x,s):=\int_{0}^s g(x,t)\,dt$, $x\in\R^N$, $t\in\R$.
\end{theorem}

We demonstrate that the Brezis-Kato result for biharmonic Laplacean as well as Theorem \ref{th:Poho} open the way to study  the existence of solutions and their regularity for \eqref{eq}. Indeed, let us assume that $g$ is independent of $x$ and the following condition holds:
\begin{itemize}
	\item[$(g0)$] there is a constant $c>0$ such that $|g(s)|\leq c\big(1+|s|^{2^{**}-1}\big)$ for $s\in\R$,
\end{itemize}
where $2\gg := \frac{2N}{N-4}$.
Then $a(x):=g(u(x))/(1+|u(x)|)\in L^{N/4}_{loc}(\R^N)$ for $u\in L^{2\gg}(\R^N)$ and in view of Theorem \ref{th:BK}, weak
solutions to the semilinear problem \eqref{eq} belong to
$C^{3,\alpha}_\text{loc}(\R^N)\cap W^{4,q}_\text{loc}(\R^N)$.
We introduce the energy functional
\begin{equation}\label{eq:action}
	J(u):=\frac12\int_{\R^N} |\Delta u|^2- \int_{\R^N} G(u)\, dx,
\end{equation}
where $G(s) = \int_0^s g(t)\d t$.
Next, we show the existence of weak solutions to \eqref{eq} under growth assumption at $0$ and at infinity inspired by a seminal paper due to Berestycki and Lions \cite{BerestyckiLions} (cf. \cite{Mederski,MederskiNon2020}). We assume that $g$ is continuous, $g(0)=0$ and $(g0)$ holds.
Let 
\[
\begin{aligned}	
	G_+(s):=
	\begin{cases}
		\int_0^s \max\{g(t),0\}\, dt  &\text{for } s\geq 0,\\
		\int_{s}^0 \max\{-g(t),0\}\, dt&\text{for }s<0,
	\end{cases}
\end{aligned}
\]
and $g_+(s)=G_+'(s)$. Suppose in addition, that
and the following conditions are satisfied:
\begin{itemize}
	\item[$(g1)$] $\lim_{s\to 0}G_+(s)/|s|^{2^{**}}=0$,
	\item[$(g2)$] there exists  $\xi_0>0$ such that $G(\xi_0)>0$,
	\item[$(g3)$] $\lim_{|s|\to \infty}G_+(s)/|s|^{2^{**}}=0$.
\end{itemize}
We introduce  the Poho\v{z}aev manifold
\begin{equation}\label{def:Poh}
\cM :=\Big\{u\in \cD^{2,2}(\R^N)\setminus\{0\}: \int_{\R^N}|\Delta u|^2=2^{**}\int_{\R^N}G(u)\, dx\Big\},
\end{equation}
and in view of Theorem \ref{th:Poho}, $\cM$ contains all nontrivial solutions.
The existence result reads as follows.
\begin{theorem}\label{thm:2}
Let $(g0)$--$(g3)$ be satisfied.
Then $\inf_{\cM} J>0$ and there is a ground state solution $u_0\in \D$ to \eqref{eq}, i.e. $u_0\in\cM$ solves \eqref{eq}
 and $J(u_0) = \inf_{\cM}J$. Moreover $u_0\in C^{3,\alpha}_\text{loc}(\R^N)\cap W^{4,q}_\text{loc}(\R^N)$, for any $0<\alpha < 1$ and $1\leq q<\infty$.
\end{theorem}

Theorem \ref{thm:2} enables us to consider
the following  nonlinearity
\begin{equation}\label{eq:logSob}
	G(s)=s^2\log |s|\quad\hbox{for }s\neq 0,\hbox{ and }G(0)=0
\end{equation}
which satisfies $(g0)$--$(g3)$. In view of Theorem \ref{thm:2} there is a ground state solution to \eqref{eq} and $$C_{N,log}:=2\gg\Big(\frac12-\frac{1}{2\gg}\Big)^{-\frac{4}{N-4}}(\inf_{\cM} J)^{\frac{4}{N-4}}.$$
We gain the following new {\em biharmonic logarithmic Sobolev inequality}.

\begin{theorem}\label{th:BihLog}
For any $u\in \cD^{2,2}(\R^N)$ such that $\int_{\R^N}|u|^2\,dx=1$,
there holds
\begin{equation}\label{logSob}
	\frac{N}{8}\log \le \le \frac{8e}{C_{N,log}(N-4)}\pr^{(N-4)/N}\int_{\R^N}|\Delta u|^2\,dx \pr\geq \int_{\R^N}|u|^2\log |u|\,dx
\end{equation}
and
$$\le \frac{8e}{C_{N,log}(N-4)}\pr^{(N-4)/N}<\Big(\frac{2}{\pi e N}\Big)^2.$$
Moreover the equality in \eqref{logSob} holds provided that $u=u_0/\|u_0\|_{L^2(\R^N)}$ and  $u_0$ is a ground state solution to \eqref{eq}. If the equality in \eqref{logSob} holds for $u$, then there are uniquely determined $\lambda>0$ and $r>0$ such that $u_0:=\lambda u(r\cdot)\in\cM$ and $u_0$ is a ground state solution to \eqref{eq}.
\end{theorem}

Recall that the classical logarithmic Sobolev inequality given in \cite{Weissler}:
\begin{equation}\label{eq:ineqLogSob}
	\frac{N}{4}\log\Big(\frac{2}{\pi e N}\int_{\R^N}|\nabla u|^2\,dx\Big)\geq  \int_{\R^N}|u|^2\log(|u|)\,dx,\quad\hbox{for }u\in H^1(\R^N), \int_{\R^N}|u|^2\,dx=1,
\end{equation}
which is equivalent to the Gross inequality \cite{Gross}, cf. \cite{LiebLoss}.
 Recall that the optimality of \eqref{eq:ineqLogSob} and the characterization of minimizers have been already proved by Carlen \cite{Carlen} in the context of the Gross inequality as well as by del Pino and Dolbeault \cite{DelPino,DelPinoJMPA} for the interpolated Gagliardo–Nirenberg inequalities and the $L^p$-Sobolev logarithmic inequality. A generalization of the optimal Gross inequality in Orlicz spaces is given by Adams \cite{Adams}. However, to the best of our knowledge, the  logarithmic Sobolev inequality for higher order operators have not been obtained in the literature so far and \eqref{eq:ineqLogSob} seems to be the first one for the biharmonic Laplacian. Note that, in contrast to \eqref{eq:ineqLogSob} and the Laplacian problem involving \eqref{eq:logSob}, we do not know ground state solutions to \eqref{eq} explicitly. Hence the exact computation of $C_{N,log}$ remains an open question.

The paper is organized as follows. In Section \ref{sec:regularity} we prove Theorem \ref{th:BK} and in Section \ref{sec:Poho} we obtain the  Poho\v{z}aev-type result. The main result of Section \ref{sec:Lions} is a general variant of Lion's lemma (Lemma \ref{lem:Lions}) in $\cD^{2,2}(\R^N)$, which is crucial for the proof of Theroem \ref{thm:2} given in Section \ref{sec:BL}. The last Section \ref{sec:BihLog} is devoted to the biharmonic logarithmic Sobolev inequality.

\section{Regularity theory and proof of Theorem \ref{th:BK}}\label{sec:regularity}

Let $N$, $k \in \N$ and $1\leq p <\infty$ with $N> kp$.
We define $\mathcal{D}^{k,p}(\R^N)$ as a completion of the space $C^\infty_0(\R^N)$ with respect to the norm
\[
\|u\|_{\cD^{k,p}} := \le \sum_{|\alpha|=k}\|D^\alpha u\|_{L^p(\R^N)}^p \pr^\frac{1}{p},\qquad u\in C^\infty_0(\R^N).
\]
Hence 
\begin{equation}\label{eq:0.1}
	\cD^{k,p}(\R^N) \subset \cD^{k-l,\frac{Np}{N-lp}}(\R^N),\qquad 0\leq l\leq k,
\end{equation}
and 
\begin{equation}\label{eq:0.2}
	\sum_{j=0}^k\sum_{|\alpha| = k -j}\|D^\alpha u \|_{L^\frac{Np}{N-jp}(\R^N)}\leq c \|u\|_{\cD^{k,p}},\qquad u \in \cD^{k,p}(\R^N).
\end{equation}

We fix an open set $\Omega\subset \R^N$.
We recall that by the standard approach based on  mollifiers and the Calderon--Zygmund $L^p$--estimates for higher order elliptic operators \cite[(2.6)]{Nirenberg} we have the following lemma.

\begin{lemma}\label{lem:1.1}
Let  $1<p<\infty$ and $k$ be a positive integer.
If $w \in L^p_{\text{loc}}(\Omega)$ and $\Delta^k w \in L^p_{\text{loc}}(\Omega)$, then $w \in W^{2k,p}_\text{loc}(\Omega)$.
\end{lemma}

Suppose that  $u\in W^{2,2}_\text{loc}(\Om)$ is a weak solution to \eqref{eq},
where $g$ satisfies \eqref{eq:BK}.
Clearly $u\in L^{2\gg}_\text{loc}(\Om)$.
Fix $U \subset\subset \Om$. Since $\frac{2N}{N+4}< \frac{N}{4}$ and $\frac{2N}{N+4}=2\gg\frac{N-4}{N+4}$, by the H\"older inequality
\[
\begin{aligned}
\int_{U}|g(x,u)|^\frac{2N}{N+4}\d x &\leq c\int_{U} |a(x)|^\frac{2N}{N+4}+|a(x)|^{\frac{N}{4}\frac{8}{N+4}}|u|^{2\gg\frac{N-4}{N+4}}\d x  < \infty,
\end{aligned}
\]
for some constant $c>0$.
Then, 
by the distributional equality
\[
\Delta^2 u = g(x,u)\in L^\frac{2N}{N+4}_\text{loc}(\Om),
\]
and Lemma \ref{lem:1.1}, we infer that $ u \in W^{4,\frac{2N}{N+4}}_\text{loc}(\Om)$. 

Now the crucial step is the following lemma.
\begin{lemma}\label{lem:2}
Let $p\geq \frac{2N}{N+4}$ and $u\in W^{4,p}_\text{loc}(\Om)$ be a weak solution to \eqref{eq}, where $g$ satisfies \eqref{eq:BK}.
Then 
\[
u\in 
\begin{cases}
L^{Np/N-5p}_\text{loc}(\Omega), &\text{if }5p < N,\\
L^q_\text{loc}(\Omega)\;\text{ for every }1\leq q < \infty ,&\text{if }5p \geq N.
\end{cases}
\]
\end{lemma}
\begin{proof}
If $4p\geq N$, then the conclusion follows immediately by the Sobolev embedding $W^{4,p}_\text{loc}(\Omega)\subset L^q_\text{loc}(\Omega)$, $q\geq 1$.
Thus, we can clearly assume that $4p< N$.
Let us define 
\begin{equation*}
\begin{aligned}
	\tilde a(x) &:= 
	\begin{cases}
		\frac{g(x,u(x))}{u(x)}\chi_{\left\{ x \in \Omega\mid |u(x)|> 1 \right\}}(x), &\text{for }u(x)\neq 0,\\
		0&\text{for }u(x)=0,
	\end{cases}\\
	b(x) &:= g(x,u(x)) \chi_{\left\{ x\in \Omega\mid |u(x)|\leq 1\right\}}(x),
\end{aligned}
\end{equation*}
and observe that $g(x,u)=\tilde a(x)u+b(x)$ and $\tilde a,b\in L^{N/4}_\text{loc}(\Omega)$.

Let $U$ be an arbitrary open bounded subset of $\Omega$ such that $U\subset \overline{U}\subset \Omega$.
We find an open bounded $V$ with $C^\infty$-smooth boundary such that $\overline{U}\subset V \subset \overline{ V} \subset \Omega$.
Indeed, let $\xi \in C^\infty_0(\Omega)$ be a smooth cut-off function such that $\xi \equiv 1 $ on $U$ and $0\leq \xi \leq 1$.
By Sard's theorem, there is a regular value $c\in (0,1)$.
Then $V = \xi^{-1}((c,1])$ is an open bounded subset with the smooth boundary $\partial V = \xi^{-1}(\{c\})$ satisfying  $\overline {U} \subset V \subset \overline{V}\subset \Omega$.

Now take $\eta \in C^\infty_0(V)$ such that $\eta = 1$ on $U$ and $0\leq \eta \leq 1$.
We restrict our problem to $V$.
By the assumption $u\in W^{4,p}(V)$ is a distributional solution of
\begin{equation}\label{eq:1.11}
\Delta^2 u = \tilde a(x)u+b(x)\qquad\text{in } V
\end{equation}
and $\tilde a,b\in L^{N/4}(V)$.
We define 
\[
v := u\eta.
\]
Certainly, we have $v\in W^{4,p}(V)\subset H^2(V)$  and $v\in H^1_0(V)$, since $\supp \eta \subset \subset V$.
Standard calculations yield
\begin{equation}\label{eq:1.12}
\begin{aligned}
\Delta^2 v &= (\Delta^2 u)\eta + \underbrace{4\nabla \Delta u \cdot \nabla \eta + 4\sum_{i=1^N}\nabla u_{x_i}\cdot \nabla \eta_{x_i} + 2 \Delta u \Delta \eta + 4 \nabla u \cdot \nabla \Delta \eta + u \Delta^2 \eta}_{=: K(u)}\\
&=: (\Delta^2 u ) \eta + K(u).
\end{aligned}
\end{equation}
Observe that $u\in W^{4,p}(V)\subset W^{3,p^\ast}(V)$, $p^\ast = \frac{Np}{N-p}$ and $\eta \in C^\infty_0(V)$ imply that 
\begin{equation}\label{eq:1.20}
\|K(u)\|_{L^{p^\ast}(V)} \leq c\|u\|_{W^{3,p^\ast}(V)}\|\eta\|_{W^{4,\infty}(V)}\leq c(\eta)\|u\|_{W^{4,p}(V)},
\end{equation}
for some constant $c(\eta)>0$.

In view of \cite[Lemma B.2]{Van_der_Vorst}, for every $\e>0$ there are  $q_\e \in L^{N/4}(V)$ and $\widehat{f}_\e \in L^\infty(V)$ such that 
\begin{equation}\label{eq:1.9}
\tilde a(x) v = q_\e (x)v + \widehat{f}_\e,
\end{equation}
and
\begin{equation}\label{eq:1.10}
\qquad \|q_\e\|_{L^{N/4}(V)}\leq \e.
\end{equation}
By \eqref{eq:1.12}, \eqref{eq:1.11} and \eqref{eq:1.9} we get
\begin{equation}\label{eq:1.13}
\begin{aligned}
\Delta^2 v &= (\Delta^2 u)\eta + K(u) \\
&= \tilde a(x)v + b(x)\eta + K(u) \\
&= q_\e(x)v + f_\e + K(u),
\end{aligned}
\end{equation}
where 
\begin{equation}\label{eq:1.15}
f_\e := \widehat{f}_\e + b(x)\eta \in L^{\frac{N}{4}} (V).
\end{equation}

We  recall some needed regularity results from \cite{Agmon_et_al} (see also \cite[Thm 2.20]{Gazzola}), for all  $1<q<\infty$, $\bar g\in L^q(V)$, there exists a unique strong solution $u \in W^{4,q}(V)$ to the problem
\[
\begin{cases}
(-\Delta)^2 u = \bar g &\text{in } V,\\
u = \Delta u = 0 &\text{on }\partial V.
\end{cases}
\]
satisfying 
\begin{equation*}
\|u\|_{W^{4,q}(V)}\leq c_q \|\bar g\|_{L^q(V)},
\end{equation*}
where $c_q>0$ depends only on $N$, $q$ and $V$.
Denote by $T_q$ the linear operator $g\mapsto u$ considered as an operator from $L^q(V)$ to  $W^{4,q}(V)$ and rewrite the above inequality as
\begin{equation}\label{eq:1.18}
\|T_q \bar g\|_{W^{4,q}(V)}\leq c_q \|\bar g\|_{L^q(V)}.
\end{equation}
Obviously, $T_q$ is the $L^q$-inverse of the bilaplacian $(-\Delta)^{2}$ considered with the Navier boundary conditions $u = \Delta u = 0$ on $\partial V$.

Now we can rephrase \eqref{eq:1.13} in the language of operators
\begin{equation}\label{eq:1.14}
v - A_{\e,q} v = h_{\e,q}, 
\end{equation}
where $A_{\e,q}v  := T_q(q_\e v)$ and $h_{\e,q} := T_q (f_\e + K(u))$.

We consider two cases separately.

\begin{center}
\textbf{Case I: $5p <  N$.}
\end{center}

In what follows we take $q= p^\ast$.
By the Sobolev embedding $W^{4,p^\ast}(V)\subset L^\frac{Np}{N-5p}(V)$, \eqref{eq:1.18}, \eqref{eq:1.15} and \eqref{eq:1.20}, we have
\begin{equation}\label{eq:1.19}
\begin{aligned}
\|h_{\e,p^\ast}\|_{L^\frac{Np}{N-5p}(V)} &\leq c_\text{Sobolev} \|T_{p^\ast}(f_\e + K(u))\|_{W^{4,p^\ast}(V)}\\
&\leq c_\text{Sobolev}c_{p^\ast}\|f_\e + K(u)\|_{L^{p^\ast}(V)}\\
&\leq c\le \|f_\e\|_{L^{\frac{N}{4}}(V)} + \|K(u)\|_{L^{p^\ast}(V)}\pr\\
&\leq c \le \|f_\e\|_{L^{\frac{N}{4}}(V)} + c(\eta)\|u\|_{W^{4,p}(V)}\pr,
\end{aligned}
\end{equation}
where  $c>0$ is some constant.
We estimate the norm of the linear operator $A_{\e,p^\ast}:L^\frac{Np}{N-5p}(V)\to L^\frac{Np}{N-5p}(V)$ applying the Sobolev embedding $W^{4,p^\ast}(V)\subset L^\frac{Np}{N-5p}(V)$ and \eqref{eq:1.18}
\begin{equation}\label{eq:1.22}
\|A_{\e,p^\ast}v\|_{L^\frac{Np}{N-5p}(V)} \leq c_\text{Sobolev} \|T_{p^\ast} (q_\e v)\|_{W^{4,{p^\ast}}(V)}\leq c_\text{Sobolev}c_{p^\ast}\|q_\e v\|_{L^{p^\ast}(V)}.
\end{equation}
We use the H\"{o}lder inequality with the exponents
\[
\frac{1}{\frac{N}{4}}+\frac{1}{\frac{Np}{N-5p}}=\frac{1}{p^\ast}
\]
to obtain
\begin{equation}\label{eq:1.21}
\|q_\e v\|_{L^{p^\ast}(V)}\leq \|q_\e\|_{L^{N/4}(V)}\|v\|_{L^\frac{Np}{N-5p}(V)}.
\end{equation}
In view of \eqref{eq:1.22}, \eqref{eq:1.21} and \eqref{eq:1.10} we gain
\[
\|A_{\e,{p^\ast}}v\|_{L^\frac{Np}{N-5p}(V)} \leq c_\text{Sobolev}c_{p^\ast}\e\|v\|_{L^\frac{Np}{N-5p}(V)}.
\]
We choose $\e := \le 2c_\text{Sobolev}c_{p^\ast}\pr ^{-1}$ to deduce
\begin{equation}\label{eq:1.23}
\|A_{\e, {p^\ast}}\|_{L^\frac{Np}{N-5p}\to L^\frac{Np}{N-5p}}\leq \frac{1}{2}.
\end{equation}
Then $(I-A_{\e,{p^\ast}})$ is invertible on the space $L^\frac{Np}{N-5p}(V)$ with the norm bounded by $2$ and by \eqref{eq:1.14}
\begin{equation}\label{eq:1.24}
v = (I- A_{\e,{p^\ast}})^{-1}h_{\e,{p^\ast}},
\end{equation}
so by the above and by \eqref{eq:1.19}
\[
\begin{aligned}
\|v\|_{L^{\frac{Np}{N-5p}}(V)} &\leq \left\|\le I-A_{\e,{p^\ast}}\pr ^{-1} \right \|_{L^\frac{Np}{N-5p}\to L^\frac{Np}{N-5p}}\|h_{\e,{p^\ast}}\|_{L^\frac{Np}{N-5p}(V)}\\
&\leq 2 c\le\|f_\e\|_{L^\infty(V)} + c(\eta) \|u\|_{W^{4,p}(V)}\pr  < \infty.
\end{aligned}
\]
Hence $v\in L^\frac{Np}{N-5p}(V)$ and, since $u = v $ on $U\subset \Omega$ and $U$ is arbitrary, we finally get $u\in L^\frac{Np}{N-5p}_\text{loc}(\Omega)$ as claimed.
This finishes the proof of Case I.

\begin{center}
\textbf{Case II: $5p \geq N$.}
\end{center}

We proceed similarly as in Case I.
Fix any $\frac{Np}{N-4p}\leq q<\infty $ and define $r:=  \frac{Nq}{N+4q}$.
Then we have $1< r < \frac{N}{4} \leq \frac{Np}{N-p}$.
We employ the Sobolev embedding $W^{4,r}(V)\subset L^q(V)$, \eqref{eq:1.18}, \eqref{eq:1.15} and \eqref{eq:1.20} to estimate
\begin{equation}\label{eq:1.25}
\begin{aligned}
\|h_{\e,r}\|_{L^q(V)}&\leq c_\text{Sobolev}\|T_r(f_\e + K(u))\|_{W^{4,r}(V)}\\
&\leq c_\text{Sobolev} c_r\|f_\e + K(u)\|_{L^r(V)}\\
&\leq c\le  \|f_\e\|_{L^{\frac{N}{4}}(V)}  + \|K(u)\|_{L^{p^\ast}}\pr \\
&\leq c \le \|f_\e\|_{L^{\frac{N}{4}}(V)} + c(\eta)\|u\|_{W^{4,p}(V)}\pr,
\end{aligned}
\end{equation}
for some constant $c>0$.
We bound the norm of $A_{\e,r}:L^q(V) \to L^q(V)$ by exploiting the Sobolev embedding $W^{4,r}(V)\subset L^q(V)$ and \eqref{eq:1.18}
\begin{equation}\label{eq:1.26}
\begin{aligned}
\|A_{\e,r}\|_{L^q(V)}\leq c_\text{Sobolev}\|T_r(q_\e v)\|_{W^{4,r}(V)}\leq c_\text{Sobolev} c_r\|q_\e v\|_{L^r(V)}.
\end{aligned}
\end{equation}
We use H\"{o}lder's inequality with exponents
\[
\frac{1}{\frac{N}{4}}+\underbrace{\frac{1}{\frac{Nr}{N-4r}}}_{=\frac{1}{q}} = \frac{1}{r}
\]
 and \eqref{eq:1.10} to obtain
\begin{equation}\label{eq:1.27}
\|q_\e v\|_{L^r(V)}\leq \|q_\e \|_{L^\frac{N}{4}(V)}\|v\|_{L^q(V)}\leq \e \|v\|_{L^q(V)}.
\end{equation}
We choose $\e = \le 2 c_\text{Sobolev}c_r\pr ^{-1}$ and from  \eqref{eq:1.26}, \eqref{eq:1.27} deduce that
\[
\|A_{\e ,r}\|_{L^q\to L^q}\leq \frac{1}{2}.
\]
As in the last part of Case I, we then show that $v\in L^q(V)$.
This implies that $u\in L^q(U)$ and, since $U\subset \Omega$ and $q\geq \frac{Np}{N-4p}$ were arbitrary, the proof of Case II is completed.
\end{proof}

\begin{altproof}{Theorem \ref{th:BK}}
Let  $u\in W^{2,2}_\text{loc}(\Om)$ be a weak solution to \eqref{eq}.
Then $u\in W^{4,\frac{2N}{N+4}}_\text{loc}(\Omega)$.
We show that $u\in L^q_\text{loc}(\Omega)$, for every $q\geq 1$.
If $N=5$ or  $N=6$, then, by Lemma \ref{lem:2}, $u\in L^q_\text{loc}(\Omega)$, for every $q\geq 1$, and we are done.
If $N> 6$, then we define $p_1 := \frac{2N}{N+4}$, $5p_1<N$, and we use Lemma \ref{lem:2}  to obtain $u\in L^\frac{Np_1}{N-5p_1}_\text{loc}(\Omega)$.
Since $\frac{Np_1}{N-5p_1} = \frac{2N}{N-6}$, 
$$p_1<p_2:=\frac{Np_1}{N-5p_1}\frac{N-6}{N+2}=\frac{2N}{N+2} <\frac{N}{4}.$$
Fix $U \subset\subset \Omega$. Observe that $p_2\frac{N+2}{8}=\frac{N}{4}$ and by the H\"older inequality
\[
\begin{aligned}
	\int_{U}|g(x,u)|^{p_2}\d x &\leq c\int_{U} |a(x)|^{p_2}\,dx+c\Big(\int_{U}|a(x)|^{p_2\frac{N+2}{8}}\,dx\Big)^{\frac{8}{N+2}}\Big(\int_{U}|u|^{\frac{Np_1}{N-5p_1}}\d x\Big)^{\frac{N-6}{N+2}} < \infty
\end{aligned}
\]
for some constant $c>0$. Therefore we get
 $\Delta ^2 u =g(x,u)\in L^{p_2}_\text{loc}(\Omega)$.
 Since $u\in W^{4,p_1}_\text{loc}(\Omega)\subset L^{p_2}_\text{loc}(\Omega)$, we use Lemma \ref{lem:1.1} to get $u\in W^{4,p_2}_\text{loc}(\R^N)$. Let $K$ be the largest natural number less than $\frac{N-4}{2}$.
We continue applying Lemma \ref{lem:2} in this fashion and get a finite sequence $(p_k)_{k=1}^{K}$ such that for $k=1,...,K$
\begin{eqnarray*}
	&&p_k:=\frac{2N}{N+6-2k},\\
	&&p_k\frac{N+6-2k}{8}=\frac{N}{4},\\
	&&p_{k+1} = \frac{Np_{k}}{N-5p_{k}}\frac{N-4-2k}{N+4-2k}, \quad\hbox{if }k\geq 1.
\end{eqnarray*}
By the definition of $K$, we get $5p_K<N$, $\frac{Np_K}{N-5p_K}\geq N$ and $u\in L^\frac{Np_K}{N-5p_K}_\text{loc}(\Omega)$.
Finally,  by Lemma \ref{lem:2} we obtain that  $u\in L^q_\text{loc}(\Omega)$, for every $q\geq 1$.
Since $\Delta^2 u =  g(x,u)\in L^q_\text{loc}(\Omega)$, for every $1\leq q <\infty$, 
by Lemma \ref{lem:1.1}, $u\in W^{4,q}_\text{loc}(\Omega)$, $q \geq 1$, so by the Sobolev embedding $u \in C^{3,\alpha}_\text{loc}(\Omega)$, for every $0<\alpha<1$.
\end{altproof}

\section{Poho\v{z}aev identity}\label{sec:Poho}

\begin{altproof}{Theorem \ref{th:Poho}}
	One can find $\f\in C^\infty (\R)$ satisfying $\f\vert_{(-\infty,1]}\equiv 1$, $\f\vert_{[2,\infty)}\equiv 0$ and $0\leq \f \leq 1$.
	For every $n\geq 1$, we define $\f_n \in C^\infty_0(\R^N)$ by $\f_n (x) := \f \left (\frac{|x|^2}{n^2}\right )$.

	By Theorem \ref{th:BK}, we may assume that $u\in C^{3,\alpha}_\text{loc}(\R^N)\cap W^{4,q}_{\text{loc}}(\R^N)$, $0<\alpha<1$, $1\leq q < \infty$, so 
	\[
	0 =\Delta^2 u - g(x,u) \quad\text{a.e. in }\R^N.
	\]
	Thus, for a.e. $x\in \R^N$ and for every $n$, we obtain
	\begin{equation}\label{eq:1.3}
		0 = (\Delta ^2 u - g(x,u)) \f_n  x \cdot \nabla u.
	\end{equation}
	The following identities hold
	\[
	g(x,u)\f_n  x \cdot \nabla u = \mathrm{div}\le \f_n G(x,u)x \pr - G(x,u)  x \cdot \nabla \f_n - N\f_n G(x,u)-\vp_n x\cdot \partial_x G(x,u) 
	\]
	and
	\[
	\Delta^2 u \f_n x \cdot \nabla u = \mathrm{div}\le \f_n (x\cdot \nabla u) \nabla \Delta u\pr - (x\cdot \nabla u ) (\nabla \f_n \cdot \nabla \Delta u) - \f_n\nabla (x\cdot \nabla u) \cdot \nabla (\Delta u).
	\]
	We transform the rightmost term of the above equation
	\[
	\begin{aligned}
		\f_n\nabla (x\cdot \nabla u) \cdot \nabla (\Delta u)&= -\f_n \Delta u\Delta (x\cdot \nabla u) + \f_n \mathrm{div }\le \Delta u\nabla (x\cdot \nabla u)\pr\\
		&=-\f_n \Delta u (2\Delta u + x\cdot \nabla \Delta u) + \mathrm{div}\le \f_n \Delta u \nabla (x\cdot \nabla u)\pr - \Delta u \nabla \f_n \cdot \nabla (x\cdot \nabla u)\\
		&= -2\f_n (\Delta u)^2 - \f_n \Delta  u x\cdot \nabla \Delta u+ \mathrm{div}\le \f_n \Delta u \nabla (x\cdot \nabla u)\pr - \Delta u \nabla \f_n \cdot \nabla (x\cdot \nabla u).
	\end{aligned}
	\]
	Finally, we rewrite the second term of the above line as follows
	\[
	\f_n \Delta  u x\cdot \nabla \Delta u = \mathrm{div}\le\f_n\frac{ (\Delta u )^2}{2} x\pr - \frac{1}{2}(\Delta u)^2\nabla \f_n \cdot x - \frac{N}{2}\f_n (\Delta u)^2.
	\]
	Putting the above identities into \eqref{eq:1.3} we get
	\[
	\begin{aligned}
		0&= -\mathrm{div}\le \f_n G(x,u)x \pr + G(x,u)  x \cdot \nabla \f_n + N\f_n G(x,u) +\vp_n x\cdot \partial_x G(x,u)\\
		 &\qquad+\mathrm{div}\le \f_n (x\cdot \nabla u) \nabla \Delta u\pr - (x\cdot \nabla u ) (\nabla \f_n \cdot \nabla \Delta u)
		-\mathrm{div}\le \f_n\le \Delta u \nabla (x\cdot \nabla u)
		  - \frac{(\Delta u )^2}{2}x\pr\pr\\
		  &\qquad - \frac{N-4}{2}\f_n (\Delta u )^2  -\frac{1}{2} (\Delta u )^2 x \cdot \nabla \f_n + \Delta u \nabla \f_n \cdot \nabla (x\cdot \nabla u)
	\end{aligned}
	\]
	or, equivalently,
	\begin{multline}\label{eq:1.4}
		\mathrm{div}\le \f_n \le G(x,u)x + \Delta u \nabla(x\cdot \nabla u) - x\cdot \nabla u \nabla \Delta u  - \frac{(\Delta u )^2}{2}x\pr \pr \\
		=  G(x,u)  x \cdot \nabla \f_n + N\f_n G(x,u)+\vp_n x\cdot \partial_x G(x,u) - (x\cdot \nabla u ) (\nabla \f_n \cdot \nabla \Delta u)\\ - \frac{N-4}{2}\f_n (\Delta u )^2  -\frac{1}{2} (\Delta u )^2 x \cdot \nabla \f_n + \Delta u \nabla \f_n \cdot \nabla (x\cdot \nabla u).
	\end{multline}
	Fix $n\geq 1 $ and take $R>0$ such that $\mathrm{supp}\, \f_n\subset B_R$.
	By the divergence theorem, we obtain
	\[
	\begin{aligned}
		0 &= \int_{B_R}G(x,u)  x \cdot \nabla \f_n + N\f_n G(x,u)+\vp_n x\cdot \partial_x G(x,u)  - (x\cdot \nabla u ) (\nabla \f_n \cdot \nabla \Delta u)\\
		&\qquad - \frac{N-4}{2}(\Delta u )^2\f_n   -\frac{1}{2} (\Delta u )^2 x \cdot \nabla \f_n + \Delta u \nabla \f_n \cdot \nabla (x\cdot \nabla u)\,dx.
	\end{aligned}
	\]
	Note that
	\[
	-\int_{B_R} (x\cdot \nabla u ) (\nabla \f_n \cdot \nabla \Delta u)\,dx = \int_{B_R}\Delta u \nabla \f_n \cdot \nabla (x \cdot \nabla u) + \Delta u \Delta \f_n x\cdot \nabla u\,dx - \underbrace{\int_{B_R}\mathrm{div}\, \le x\cdot \nabla u \Delta u \nabla \f_n \pr}_{=0}\,dx.
	\]
	Summing up, we have
	\begin{equation}\label{eq:1.5}
		\begin{aligned}
			0 &= \int_{B_R}G(u)  x \cdot \nabla \f_n + N\f_n G(x,u)+\vp_n x\cdot \partial_x G(x,u) + 2 \Delta u \nabla \f_n \cdot \nabla (x \cdot \nabla u) + \Delta u \Delta \f_n x\cdot \nabla u\\
			&\qquad\qquad - \frac{N-4}{2}(\Delta u )^2\f_n   -\frac{1}{2} (\Delta u )^2 x \cdot \nabla \f_n \,dx\\
			&= \int_{\R^N}G(u)  x \cdot \nabla \f_n + N\f_n G(x,u)+\vp_n x\cdot \partial_x G(x,u) + 2 \Delta u \nabla \f_n \cdot \nabla (x \cdot \nabla u) + \Delta u \Delta \f_n x\cdot \nabla u\\
			&\qquad\qquad - \frac{N-4}{2}(\Delta u )^2\f_n   -\frac{1}{2} (\Delta u )^2 x \cdot \nabla \f_n\,dx.
		\end{aligned}
	\end{equation}
	We return to \eqref{eq:1.5} and pass to the limit as $n\to \infty$ to obtain
	\[
	0 =N  \int_{\R^N}G(x,u)\,dx +\int_{\R^N}x\cdot \partial_x G(x,u)\,dx  - \frac{N-4}{2}\int_{\R^N}|\Delta u|^2\,dx,
	\]
	where we used Lebesgue's dominated convergence theorem and the properties of $\f_n$. The proof is completed.
\end{altproof}

\section{Lions lemma}\label{sec:Lions}
We prove a biharmonic variant of Lion's lemma, cf. \cite{Lions1,Lions2}, \cite[Section 2]{Mederski}.
\begin{lemma}\label{lem:Lions}
	Suppose that $(u_n)$ is bounded in $\cD^{2,2}(\R^N)$ and for some $r>0$ 	
	\begin{equation}\label{eq:LionsCond11}
		\lim_{n\to\infty}\sup_{y\in \R^N} \int_{B(y,r)} |u_n|^2\,dx=0.
	\end{equation}
	Then  
	$$\int_{\R^N} \Psi(u_n)\, dx\to 0\quad\hbox{as } n\to\infty$$
	for every continuous $\Psi:\R\to \R$ satisfying
	\begin{equation}
		\label{eq:Psi}
		\lim_{s\to 0} \frac{\Psi(s)}{|s|^{2^{**}}}=\lim_{|s|\to\infty} \frac{\Psi(s)}{|s|^{2^{**}}}=0.
	\end{equation}
\end{lemma}

We prove the following result, which implies the variant of Lions's lemma in $\D$.
\begin{lemma}\label{lem:Conv}
	Suppose that   $(u_n)\subset \D$ is bounded. Then $u_n(\cdot+y_n)\weakto 0$ in $\D$ for any $(y_n)\subset \Z^N$ if and only if
	\begin{equation*}
		\int_{\R^N} \Psi(u_n)\, dx\to 0\quad\hbox{as } n\to\infty
	\end{equation*}
	for any continuous $\Psi:\R\to \R$ satisfying \eqref{eq:Psi}.
\end{lemma}
\begin{proof}
	Let $(u_n)$ be a sequence in $\D$ be such that $u_n(\cdot+y_n)\weakto 0$ in $\D$ for every $(y_n)\subset \Z^N$.
	Take any $\eps>0$ and $2^\ast<p<2\gg$ and suppose that $\Psi$ satisfies \eqref{eq:Psi}. Then we find $0<\delta<M$ and $c(\eps)>0$ such that 
	\begin{eqnarray*}
		\Psi(s)&\leq& \eps |s|^{2\gg}\quad\hbox{ for }|s|\leq \delta,\\
		\Psi(s)&\leq& \eps |s|^{2\gg}\quad\hbox{ for }|s|>M,\\
		\Psi(s)&\leq& c(\eps) |s|^{p}\quad\hbox{ for }|s|\in (\delta,M].
	\end{eqnarray*}
	Let us define $(w_n)$ by
	\[
	w_n(x):=
	\begin{cases}|u_n(x)| &\text{for }|u_n(x)|>\delta,\\
		|u_n(x)|^{2\gg/2^\ast}\delta ^{1-2\gg/2^*} &\text{for }|u_n(x)|\leq \delta.
	\end{cases}
	\] 
	We are about to show that $(w_n)$ is bounded in $W^{1,2^\ast}(\R^N)$.
	First of all, we have
	\begin{equation}\label{eq:1.30}
		\begin{aligned}
			\int_{\R^N} |w_n(x)|^{2^\ast}\d x &= \int_{\{|u_n|\leq \delta\}}\delta ^{2^\ast - 2\gg} |u_n|^{2\gg}\d x + \int_{\{|u_n|\geq \delta\}}|u_n|^{2^\ast}\d x\\
			&= \delta ^{2^\ast - 2\gg}\int_{\{|u_n|\leq \delta\}}|u_n|^{2\gg}\d x + \int_{\{|u_n|> \delta\}}\frac{|u_n|^{2\gg}}{|u_n|^{2\gg-2^\ast}}\d x\\
			&\leq \delta ^{2^\ast - 2\gg}\int_{\{|u_n|\leq \delta\}}|u_n|^{2\gg}\d x + \int_{\{|u_n|> \delta\}}\frac{|u_n|^{2\gg}}{\delta^{2\gg-2^\ast}}\d x \\
			&= \delta^{2^\ast-2\gg} \int_{\R^N}|u_n|^{2\gg}\d x.
		\end{aligned}	
	\end{equation}
	By the absolute continuous characterization (see \S 1.1.3 in \cite{Mazja}), we infer that each $u_n$ is absolutely continuous on almost every line parallel to the $0x_i$-axis, for $i=1,\ldots, N$. Thus the same holds for each $w_n$, since $w_n = F(u_n)$, where $F(t) = \min\{\delta^{1-2\gg/2^\ast}|t|^{2\gg/2^\ast},|t|\}$ is a globally Lipschitz function.
	Moreover, for every $i=1,\ldots, N$, we have
	\[
	\frac{\partial w_n}{\partial x_i} = 
	\begin{cases}
		\frac{2\gg}{2^\ast}\delta^{1-2\gg/2^\ast}\mathrm{sign}(u_n)|u_n|^{2\gg/2^\ast-1}\frac{\partial u_n}{\partial x_i}	, &\text{for } |u_n(x)|\leq \delta,\\
		\mathrm{sign}(u_n) \frac{\partial u_n}{\partial x_i},&\text{for }|u_n(x)|> \delta.	
	\end{cases}
	\]
	Thus
	\begin{equation}\label{eq:1.31}
		\begin{aligned}
			\int_{\R^N}\left | \frac{\partial w_n}{\partial x_i}\right|^{2^\ast}\d x &= \le \frac{2\gg}{2^\ast}\pr ^{2^\ast}\delta^{2^\ast-2\gg}\int_{\{|u_n|\leq \delta\}} |u_n|^{2\gg- 2^\ast}\left |\frac{\partial u_n}{\partial x_i}\right|^{2^\ast} \d x + \int_{\{|u_n|>\delta\}}\left|\frac{\partial u_n}{\partial x_i} \right |^{2^\ast}\d x\\
			&\leq \le \frac{2\gg}{2^\ast}\pr ^{2^\ast}\int_{\{|u_n|\leq \delta\}} \left |\frac{\partial u_n}{\partial x_i}\right|^{2^\ast} \d x + \int_{\{|u_n|>\delta\}}\left|\frac{\partial u_n}{\partial x_i} \right |^{2^\ast}\d x\\
			&\leq \le \frac{2\gg}{2^\ast}\pr ^{2^\ast} \int_{\R^N}\left| \frac{\partial u_n}{\partial x_i}\right|^{2^\ast} \d x.
		\end{aligned}	
	\end{equation}
	By \eqref{eq:1.30}, \eqref{eq:1.31} (again using an absolute continuous characterization on lines from \S 1.1.3 \cite{Mazja}) and the fact that $(u_n)$ is bounded in $\D$, we conclude that $(w_n)$ is bounded in $W^{1,2^\ast}(\R^N)$.
	
	Let  $\Om=(0,1)^N$ and $y\in  \R^N$ be arbitrary.
	Then, by the Sobolev inequality one has 
	\begin{eqnarray*}
		\int_{\Om+y}\Psi(u_n)\d x&=&
		\int_{(\Om+y)\cap\{\delta<|u_n|\leq M\}}\Psi(u_n)\d x
		+\int_{(\Om+y)\cap (\{|u_n|> M\}\cup \{|u_n|\leq \delta\})}\Psi(u_n)\d x\\
		&\leq &
		c(\eps)\int_{(\Om+y)\cap\{\delta<|u_n|\leq M\}}|w_n|^p\d x
		+\eps \int_{(\Om+y)\cap (\{|u_n|> M\}\cup \{|u_n|\leq \delta\})}|u_n|^{2\gg}\d x\\
		&\leq &c(\eps) C\Big(\int_{\Om+y}|w_n|^{2^\ast} + |\nabla w_n|^{2^\ast}\d x \Big)\Big(\int_{\Om+y}|w_n|^{p}\d x\Big)^{1-2^{\ast}/p}+\eps \int_{\Om+y}|u_n|^{2\gg}\d x,
	\end{eqnarray*}
	where $C>0$ is a constant from the Sobolev inequality.
	Then we sum the inequalities over $y\in\Z^N$ and get  
	$$\begin{aligned}
		\int_{\R^N}\Psi(u_n)\d x	&\leq c(\eps) C\left(\int_{\R^N}|w_n|^{2^\ast} + |\nabla w_n|^{2^\ast}\d x \right)\left(\sup_{y\in\Z^N}\int_{\Om}|w_n(\cdot +y)|^{p}\d x\right)^{1-2^\ast/p}+\eps \int_{\R^N}|u_n|^{2\gg}\d x.
	\end{aligned}$$
	Let us take $(y_n)\subset\Z^N$ such that
	$$\sup_{y\in\Z^N}\int_{\Om}|w_n(\cdot+y)|^{p}\d x\leq 2\int_{\Om}|w_n(\cdot+y_n)|^{p}\d x$$
	for any $n\geq 1$.
	By the assumption
	$u_n(\cdot+y_n)\weakto 0$ in $\D$ and passing to a subsequence  we obtain  $u_n(\cdot+y_n)\to 0$ in $L^p(\Om)$.
	
	Since $|w_n(x)|\leq |u_n(x)|$,  we infer that $w_n(\cdot+y_n)\to 0$ in $L^p(\Om)$. 
	Therefore 
	$$\limsup_{n\to\infty}\int_{\R^N}\Psi(u_n)\d x\leq \eps \limsup_{n\to\infty}\int_{\R^N}|u_n|^{2\gg}\d x,$$
	and since $\eps>0$ is arbitrary, the assertion follows.
	
	On the other hand, suppose that  $u_n(\cdot+y_n)$ does not converge to $0$  in $\D$, for some $(y_n)$  in $\Z^N$, and $\Psi(u_n)\to 0$ in $L^1(\R^N)$. We may assume that $u_n(\cdot+y_n)\to u_0\neq 0$ in $L^p(\Om)$ for some bounded domain $\Om\subset\R^N$ and $1<p<2\gg.$ Take any $\eps>0$, $q>2\gg$ and let us define $\Psi(s):=\min\{|s|^p,\eps^{p-q}|s|^q\}$ for $s\in\R$. Then
	\begin{eqnarray*}
		\int_{\R^N} \Psi(u_n)\d x&\geq& \int_{\Om + y_n\cap \{|u_n|\geq \eps\}}|u_n|^p\d x+\int_{\Om+ y_n\cap \{|u_n|\leq \eps\}} \eps^{q-p}|u_n|^q\d x\\
		&=& \int_{\Om+y_n} |u_n|^p\d x+\int_{\Om+y_n\cap \{|u_n|\leq \eps\}} \eps^{p-q}|u_n|^q-|u_n|^p\d x\\
		&\geq& \int_{\Om+y_n} |u_n|^p\d x - \e^p|\Omega|.\\
	\end{eqnarray*}
	Thus we get $u_n(\cdot + y_n) \to 0$ in $L^p(\Om)$ and this contradicts $u_0\neq 0$.
\end{proof}

\begin{altproof}{Lemma \ref{lem:Lions}}
	Suppose that there is $(y_n)\subset \Z^N$ such that $u_n(\cdot+y_n)$ does not converge weakly to $0$ in $\D$. Since  $u_n(\cdot+y_n)$ is bounded, there is $u_0\neq 0$ such that, up to a subsequence,
	$$u_n(\cdot+y_n)\weakto u_0\quad \text{in }\D,$$
	as $n\to\infty$. We find $y\in \R^N$ such that $u_0\chi_{B(y,r)}\neq 0$ in $L^2(B(y,r))$. 
	Observe that, passing to a subsequence, we may assume that $u_n(\cdot+y_n)\to u_0$ in $L^2(B(y,r))$. Then, in view of \eqref{eq:LionsCond11}
	$$\int_{B(y,r)} |u_n(\cdot+y_n)|^2\,dx=\int_{B(y_n+y,r)} |u_n|^2\,dx\to 0$$
	as $n\to\infty$, which contradicts the fact $u_n(\cdot+y_n)\to u_0\neq 0$ in $L^2(B(y,r))$. Therefore $u_n(\cdot+y_n)\weakto 0$ in $\D$ for any $(y_n)\subset \Z^N$ and by Lemma \ref{lem:Conv} we conclude.
\end{altproof}

\section{Proof of Theorem \ref{thm:2}}\label{sec:BL}

In this section we adapt a variational approach from \cite[Section 3]{Mederski} for the bi-Laplacian.
Let
\[
\begin{aligned}	
	G_-(s):=
	\begin{cases}
		\int_0^s \max\{-g(t),0\}\, dt  &\text{for } s\geq 0,\\
		\int_{s}^0 \max\{g(t),0\}\, dt&\text{for }s<0.
	\end{cases}
\end{aligned}
\]
Notice that $G_+$, $G_-\geq 0$ and $G= G_+ - G_-$.

First, we sketch our approach with an approximation $J_\eps$ of $J$ and present some auxiliary lemmas.
The proof of Theorem \ref{thm:2} is postponed to the end of the section.
Let
\[
g_{+}(s):=G_+^\prime (s) \quad\text{and}\quad g_{-}(s):=g_{+}(s)-g(s), \qquad s\in \R.
\]
Notice that $G_{-}(s) = \int_0^sg_-(t)\d t\geq 0$, for $s\in\R$.
In view of (g1) and (g3), there is some $c>0$ such that for every $s\in \R$
\begin{equation}\label{eq:1.39}
	|G_+(s)|\leq c|s|^{2\gg},
\end{equation}
so $G_{+}(u)\in L^{1}(\R^N)$ whenever $u\in\cD^{2,2}(\R^N)\subset L^{2^{**}}(\R^N)$.
On the other hand, $G_{-}(u)$ may not be integrable, for $u\in \D$, unless $G_{-}(u)\leq c|u|^{2^{**}}$ for some $c>0$. To overcome this problem, for $\eps\in (0,1)$, we define $\vp_\eps:\R\to [0,1]$ by 
\[
\vp_\eps(s):=
\begin{cases}\frac{1}{\e^{2\gg-1}}|s|^{2\gg-1} &\text{for } |s|\leq \e,\\
	1 &\text{for }|s|\geq \e.
\end{cases}
\]
We introduce a new functional
\begin{equation}\label{eq:actionEPS}
	J_\eps(u):=\frac12\int_{\R^N} |\Delta u|^2+\int_{\R^N} G_{-}^\e(u)\, dx-\int_{\R^N} G_{+}(u)\, dx,
\end{equation}
where $G_{-}^\e(s) := \int_0^s \f_\e(t)g_{-}(t)\d t$, $s\in \R$.
By (g0), 
 there is $c(\e)>0$ such that
\begin{equation}\label{eq:1.32}
	|\f_\e(s)g_{-}(s)|\leq c(\e) |s|^{2\gg -1},\quad s\in\R.
\end{equation}
This implies that $G_{-}^{\e}(s)\leq c(\eps)|s|^{2^{**}}$ for any $s\in\R$ and some constant $c(\eps)>0$ depending on $\eps>0$. Hence, for $\eps\in (0,1)$,  $J_\eps$ is well-defined on $\cD^{2,2}(\R^N)$, continuous and $J_\eps'(u)(v)$ exists for any $u\in\cD^{2,2}(\R^N)$ and $v\in\cC_0^\infty(\R^N)$. Therefore, we say that $u$ is a {\em critical point} of $J_\eps$ provided that $J_\eps'(u)(v)=0$ for any $v\in\cC_0^\infty(\R^N)$.

We define, for $\e\in (0,1)$,
\begin{eqnarray*}	
	G_\e &:=& G_+ - G^\e _-,\\	
	\cM_\eps&:=&\Big\{u\in \D\setminus\{0\}: \int_{\R^N}|\Delta u|^2-2\gg\int_{\R^N}G_\eps(u)\, dx=0\Big\},\\
	\cP_\e&:=&\Big\{u\in \D: \int_{\R^N}G_\eps(u)\, dx>0\Big\}\neq \emptyset,\\
	c_\e&:=&\inf_{u\in \cM_\e}J_\e (u).
\end{eqnarray*}
and introduce the map $m_{\cP_\e}:\cP_\e \to \cM _\e$ given by
\[
m_{\cP_\e}(u) = u(r_\e\cdot),
\]
where 
\[
r_\e= r_\e(u) := \le\frac{ 2\gg\int_{\R^N}G_\e(u)\d x}{\int_{\R^N}|\Delta u|^2}\pr ^{1/4} = \frac{\le 2\gg\int_{\R^N}G_\e(u)\d x\pr^\frac{1}{4}}{\|u\|^{1/2}}.
\]
We check that $m_{\cP_\e}$ is well-defined.
If $u\in \cP_\e$, then
\[
\begin{aligned}
	\int_{\R^N}|\Delta (m_{\cP_\e}(u)(x)|^2\d x &= r_\e^{4-N}\int_{\R^N}|\Delta u|^2\d x\\
	&=\le 2\gg\int_{\R^N}G_\e(u)\d x\pr^\frac{4-N}{4}\|u\|^\frac{N-4}{2}\|u\|^2\\
	&= \le 2\gg\int_{\R^N}G_\e(u)\d x \pr \frac{\|u\|^\frac{N}{2}}{\le 2\gg\int_{\R^N}G_\e(u)\d x\pr^\frac{N}{4}}\\
	&= \le 2\gg\int_{\R^N}G_\e(u)\d x \pr r_\e^{-N} \\
	&= 2\gg\int_{\R^N}G_\e(m_{\cP_\e}(u)(x))\d x .
\end{aligned}
\]
\hfill$\square$
\begin{lemma} \label{lemma:ineq}
	For every $\delta>0$ there is $c_\delta>0$ such that
	$$G_\eps(u+v)-G_\eps(u)-\delta |u|^{2\gg}\leq c_\delta |v|^{2\gg}$$
	for all $u,v\in\R$.
\end{lemma}
\begin{proof}
	First, we show that for every $\delta>0$ there is $c(\delta)>0$ such that
	\begin{equation}\label{eq:1.33}
		|G^\e_-(u+ v) - G^\e_-(u)|\leq \delta |u|^{2\gg} + c(\delta) |v|^{2\gg},\quad u,\;v\in \R.
	\end{equation}
	Fix $\delta>0$ and $u$, $v\in \R$.
	By the mean value theorem, there is $\theta \in (0,1)$ such that
	\[
	\begin{aligned}
		|G^\e_-(u+v) - G^\e_-(u)| &\leq  |\f_\e(u+\theta v)g_-(u+\theta v)||v|\\
		&\leq c(\e)|u+\theta v|^{2\gg-1}|v|\\
		&\leq c_1(\e)|u|^{2\gg-1}|v| + c_1(\e)|v|^{2\gg},
	\end{aligned}
	\] 
	where we used \eqref{eq:1.32}.
	We exploit the Young inequality with $\delta/c_1(\e)$
	\[
	|u|^{2\gg-1}|v| \leq \frac{\delta}{c_1(\e)}|u|^{(2\gg-1)p} + c_2(\delta,\e)|v|^q,\quad \text{where }p =\frac{2\gg}{2\gg-1},\; q = 2\gg,
	\]
	to obtain
	\[
	|G^\e_-(u+v) - G^\e_-(u)| \leq \delta|u|^{2\gg} + c_3(\delta,\e)|v|^{2\gg},
	\]
	what proves the assertion.
	
	Now, we show that for every $\delta>0$ there is $c(\delta)>0$ such that 
	\[
	G_+(u+v) - G_+(u) - \delta|u|^{2\gg} \leq c(\delta)|v|^{2\gg},\qquad u,\;v\in \R.
	\]
	Fix $\delta>0$ and $u$, $v\in \R$.
	By (g1) and (g3), there are $0<\eta <M$ such that 
	\[
	G_+(s)\leq \frac{2}{2^{2\gg}}\delta |s|^{2\gg},
	\]
	if $|s|<\eta$ or $|s|>M$.
	We consider four cases.
	\\	\textit{Case I: $|u+v|<\eta$ or $|u+v|>M$.}\\
	We use the fact that $G_+ \geq 0$ and obtain
	\[
	G_+(u+v) - G_+(u)\leq G_+(u+v)\leq \frac{2}{2^{2\gg}}\delta|u+v|^{2\gg}\leq \delta \le |u|^{2\gg} + |v|^{2\gg} \pr,
	\]
	what proves the assertion.
	\\	\textit{Case II: $\eta \leq|u+v|\leq M$ and $|v|>M$.}\\
	There is $c>0$ such that $G_+(s)\leq c|s|^{2\gg}$, for every $s\in\R$, so
	\[
	G_+(u+v) - G_+(u)\leq G_+(u+v)\leq c|u+v|^{2\gg}\leq cM^{2\gg} \leq c|v|^{2\gg}
	\]	
	and we are done.
	\\	\textit{Case III: $\eta\leq |u+v|\leq M$ and $\eta/2 \leq |v|\leq M$.}\\
	The set $C:= \left\{(u,v)\in\R^2\mid \eta\leq |u+v|\leq M\text{ and } \eta/2 \leq |v|\leq M\right\}$ is compact and the function $h:C\to \R$, given by $h(u,v) := \frac{G_+(u+v) - G_+(u) - \delta|u|^{2\gg}}{|v|^{2\gg}}$, is continuous.
	Thus, there is $c(\delta)>0$ such that $\max_{(u,v)\in C}h(u,v)\leq c(\delta)$ and we are done.
	\\ \textit{Case IV: $\eta\leq |u+v|\leq M$ and $|v| < \eta /2$.}\\
	By the continuity of $g_+$ and by (g0), there is $c(\eta)$ such that 
	\[
	|g_+(s)|\leq c(\eta)|s|^{2\gg-1}, \qquad |s|\geq \frac{\eta}{2}.
	\]
	By the mean value theorem, there is $\theta\in (0,1)$ such that
	\[
	G_+(u+v) - G_+(u) = g_+(u+\theta v)v.
	\]
	Notice that $|u+\theta v| \geq |u + v| - (1-\theta)|v|	>\eta - \eta/2 = \eta/2$, so combining the above we obtain
	\[
	G_+(u+v) - G_+(u) \leq c(\eta)|u + \theta v|^{2\gg-1}|v|.
	\]
	We then proceed as in the first part of the proof.
	
	Finally, we use the above results to deduce
	\[
	\begin{aligned}
		G_\e(u+v) - G_\e(u) &=  G_+(u+v) - G_+(u) - \le G^\e_-(u+v) -G^\e_-(u)\pr \\
		&\leq \delta|u|^{2\gg}+ c(\delta)|v|^{2\gg} + |G^\e_-(u+v) -G^\e_-(u)| \\
		&\leq 2\le \delta|u|^{2\gg}+ c(\delta)|v|^{2\gg}\pr.
	\end{aligned}
	\]
\end{proof}

\begin{lemma}\label{lem:theta}
	Suppose that $(u_n)\subset \cM_\eps$, $J_\eps(u_n)\to c_\eps$ and
	$$u_n\weakto\tu\neq 0\hbox{ in } \D,\;u_n(x)\to\tu(x)\quad\hbox{ for a.e. }x\in\R^N$$ for some  $\tu\in \D$. Then $u_n\to\tu$, $\tu$  is a critical point of $J_\eps$ and $J_\eps(\tu)=c_\eps$.
\end{lemma}
\begin{proof}
	It follows, by Lemma \ref{lemma:ineq}, that for every $\delta>0$ theres is $c(\delta)>0$ such that 
	\[
	|G_\e (u+v) - G_\e(u)| \leq \delta |u|^{2\gg} + c(\delta)|v|^{2\gg},\quad u,\; v \in \R.
	\]
Thus taking any $v\in \cC_0^{\infty}(\R^N)$ and $t\in\R$ we observe that $(G_\eps(u_n+tv)-G_\eps(u_n))$ is uniformly integrable and tight. In view of Vitali's convergence theorem we have
	\[
	\lim_{n\to \infty} \int _{\R^N} G_\e(u_n + t v ) - G_\e (u_n)\d x = \int _{\R^N} G_\e(\tu + t v ) - G_\e (\tu)\d x.
	\]	
	Since each $u_n\in \cM_\e$, we get
	\[
	c_\e \leftarrow J_\e (u_n) = \frac{1}{2}\int_{\R^N}|\Delta u_n|^2 \d x- \int_{\R^N}G_\e(u_n)\d x = \le \frac{2\gg}{2}-1\pr\int_{\R^N}G_\e(u_n)\d x,
	\]
	so 
	\begin{equation}\label{eq:1.34}
		A:=\lim_{n\to \infty}\int_{\R^N} G_\e(u_n)\d x = \frac{1}{2\gg}\le \frac{1}{2}-\frac{1}{2\gg} \pr^{-1} c_\e>0.
	\end{equation}
	Combining the above we have
	\begin{equation}\label{eq:1.37}
		\begin{aligned}
			\lim_{n\to\infty}\int_{\R^N}G_\eps(u_n+tv)\,dx &=\lim_{n\to\infty}\int_{\R^N}G_\eps(u_n)\,dx+\int_{\R^N}G_\eps(\tu+tv)\,dx-\int_{\R^N}G_\eps(\tu)\,dx\\
			& = A +\int_{\R^N}G_\eps(\tu+tv)\,dx-\int_{\R^N}G_\eps(\tu)\,dx.
		\end{aligned}	
	\end{equation}
	By \eqref{eq:1.34} and Lemma \ref{lemma:ineq}, $u_n + t v \in \cP_\e$ for sufficiently large $n$ and sufficiently small $|t|$.
	Thus and by \eqref{eq:1.37}, for sufficiently small $|t|$, we have
	\[
	\begin{aligned}
		\lim_{n\to\infty}\frac{1}{t} \le \le \int_{\R^N}G_\e(u_n + tv)\d x\pr ^\frac{N-4}{N} - \le \int_{\R^N}G_\e(u_n)\d x\pr^\frac{N-4}{N} \pr \\
		= \frac{1}{t}\le \le A +\int_{\R^N}G_\e(\tu + tv)\d x - \int_{\R^N}G_\e(\tu)\d x\pr^\frac{N-4}{N} - A^\frac{N-4}{N}\pr .
	\end{aligned}
	\]
	and, consequently, by the Lebesgue dominated convergence theorem
	\begin{equation}\label{eq:thetain1}
			\lim_{t\to 0}\frac{1}{t}\le \le A +\int_{\R^N}G_\e(\tu + tv)\d x - \int_{\R^N}G_\e(\tu)\d x\pr^\frac{N-4}{N} - A^\frac{N-4}{N}\pr = \frac{N-4}{N}A^\frac{-4}{N}\int_{\R^N}g_\e(\tu)v\d x,
	\end{equation}
	where $g_\e := G_\e^\prime = g_+ - \f_\e g_-$.
	
	If $u_n+tv\in\cP_\e$, then $J_\eps(m_{\cP_\e}(u_n+tv))\geq c_\eps$, so 
	$$r_\e(u_n+tv)^{4-N}\le\frac{1}{2}-\frac{1}{2\gg}\pr \int_{\R^N}|\Delta (u_n+tv)|^2\, dx\geq c_\eps.$$
	Raising both sides to the $4/N$-power yields
	\begin{equation}\label{eq:1.35}
		\le\frac{1}{2}-\frac{1}{2\gg}\pr^{\frac4N}
		\int_{\R^N}|\Delta (u_n+tv)|^2\, dx\geq	c_\eps^{\frac4N}\le 2\gg\int_{\R^N}G_\eps(u_n+tv)\, dx\pr ^{\frac{N-4}{N}}.
	\end{equation}
	Assumptions $u_n\in\cM_\e$ and $J_\e(u_n)\to c_\e$ imply that 
	\begin{equation}\label{eq:1.36}
		\begin{aligned}
			\int_{\R^N} |\Delta u_n|^2\d x \to c_\e\le \frac{1}{2}-\frac{1}{2\gg}\pr^{-1}.
		\end{aligned}
	\end{equation}
	For all $n$ and $t$, we have
	\[
	\int_{\R^N}\Delta u_n \Delta v\d x+\frac{t}{2}\int_{\R^N}|\Delta v|^2\d x = \frac{1}{2t}\le \int_{\R^N}|\Delta (u_n+tv)|^2\d x -\int_{\R^N} |\Delta u_n|^2\d x\pr.
	\]
	Hence, by \eqref{eq:1.35} and since $u_n\in \cM_\e$, for $t>0$,
	\[
	\begin{aligned}
		&\int_{\R^N}\Delta u_n \Delta v\d x+\frac{t}{2}\int_{\R^N}|\Delta v|^2\d x \geq \\
		& \frac{1}{2t}\le c_\e^\frac{4}{N}\le \frac{1}{2}-\frac{1}{2\gg}\pr^\frac{-4}{N}\le 2\gg\int_{\R^N}G_\e(u_n+tv)\d x \pr^\frac{N-4}{N}  -\le 2\gg\int_{\R^N}G_\e(u_n)\d x\pr ^\frac{N-4}{N}\le \int_{\R^N}|\Delta u_n|^2 \d x\pr^\frac{4}{N}\pr .
	\end{aligned}
	\]
	Letting $n\to \infty$, by \eqref{eq:1.37}, \eqref{eq:1.34} and \eqref{eq:1.36}, we deduce that, for sufficiently small $t>0$,
	\[
	\begin{aligned}
		&\int_{\R^N}\Delta \tu\Delta v\d x+\frac{t}{2}\int_{\R^N}|\Delta v|^2\d x \\ 
		&\qquad \geq\frac{1}{2t}c_\e^\frac{4}{N}\le \frac{1}{2}-\frac{1}{2\gg}\pr^\frac{-4}{N}\le \le2\gg\le A + \int_{\R^N}G_\e(\tu +tv)\d x - \int _{\R^N}G_\e(\tu)\d x\pr \pr^\frac{N-4}{N} - \le 2\gg A\pr ^\frac{N-4}{N}\pr\\
		&\qquad =\frac{ 2\gg}{2}A^\frac{4}{N}\frac{1}{t}\le\le A + \int_{\R^N}G_\e(\tu +tv)\d x - \int _{\R^N}G_\e(\tu)\d x\pr^\frac{N-4}{N} -   A ^\frac{N-4}{N} \pr.
	\end{aligned}
	\]	
	We pass to the limit as $t\to 0^+$ and use \eqref{eq:thetain1} to get
	\[
	\int_{\R^N}\Delta\tu \Delta v\d x \geq \frac{2\gg}{2}\frac{N-4}{N}\int_{\R^N}g_\e(\tu)v\d x =\int_{\R^N}g_\e(\tu)v\d x.
	\]
	Since $v\in C^\infty_0(\R^N)$ was arbitrary we infer that $\tu$ is a critical point of $J_\e$.
	We use the Poho\v{z}aev identity Theorem \ref{th:Poho} to the equation $\Delta^2 u  = g_\e(u)$ with $G_\e\in L^1(\R^N)$, to deduce that $\tu\in \cM_\e$, what leads to
	\[
	c_\e \leq J_\e(\tu) = \le \frac{1}{2}-\frac{1}{2\gg}\pr\int_{\R^N}|\Delta \tu|^2 \d x\leq \liminf_{n\to\infty}\le \frac{1}{2}-\frac{1}{2\gg}\pr\int_{\R^N}|\Delta u_n|^2 \d x= \lim_{n\to\infty}J_\e(u_n) = c_\e,
	\]
	where the weak l.s.c of the norm was used.
	Thus, $J_\e(\tu) = c_\e$ and $\|u_n\|\to\|\tu\|$, so $u_n\to \tu$ in $\D$.
\end{proof}

\begin{proof}[Proof of Theorem \ref{thm:2}]
	Take a minimizing sequence $(u_n)$ in $\cM_\e$ of $J_\e$, i.e., $J_\e(u_n)\to c_\e$.
	Since $u_n\in \cM_\e$, $n\geq 1$, we have 
	\[
	J_\e(u_n) = \le \frac{1}{2}-\frac{1}{2\gg}\pr\int_{\R^N}|\Delta u_n|^2\d x \to c_\e,
	\]
	and so $(u_n)$ is bounded in $\D$.
	Moreover, we have 
	\[
	2\gg \int_{\R^N}G _+(u_n)\d x \geq \int_{\R^N}|\Delta u_n|^2\d x\to \le \frac{1}{2}-\frac{1}{2\gg}\pr c_\e.
	\]
	By the assumption $G_+$ satisfies \eqref{eq:Psi}, so \eqref{eq:LionsCond11} is not satisfied.
	Passing to a subsequence, we may choose $(y_n)$ in $\R^N$ and $0\neq u_\e \in \D$ such that
	\[
	u_n(\cdot + y_n) \rightharpoonup u_\e \quad\text{in }\D,\quad u_n(x+y_n) \to u_\e(x) \quad \text{for a.e. }x\in \R^N,
	\]
	as $n\to \infty$.
	In view of Lemma \ref{lem:theta}, $u_\e\in \cM_\e$ is a critical point of $J_\e$ at the level $c_\e$.
	
	Choose $\e_n\to 0^+$.
	Fix an arbitrary $u\in \cM$.
	Since $G_\e(s) \geq G(s)$, for all $s\in \R$ and $\e \in (0,1)$, we deduce that
	\[
	\int_{\R^N}G_{\e_n}(u)\d x \geq\int_{\R^N}G(u)\d x = \frac{1}{2\gg}\int_{\R^N}|\Delta u|^2 \d x >0,
	\]
	so $m_{\cP_{\e_n}}(u)\in \cM_{\e_n}$ is well-defined.
	We have 
	\[
	\begin{aligned}
		J_{\e_n}(u_{\e_n}) &\leq J_{\e_n}(m_{\cP_{\e_n}}(u)) = \le \frac{1}{2}-\frac{1}{2\gg}\pr \underbrace{\le\frac{2\gg\int_{\R^N}G_{\e_n}(u)\d x}{\int_{\R^N}|\Delta u|^2 \d x}\pr^\frac{4-N}{4}}_{r_{\e_n}(u)^{4-N}}\int_{\R^N}|\Delta u|^2\d x\\
		&=\le\frac{1}{2} - \frac{1}{2\gg} \pr\le \int_{\R^N}|\Delta u|^2\d x\pr^\frac{N}{4}\le 2\gg\int_{\R^N}G_{\e_n}(u)\d x \pr ^{-\frac{N-4}{4}}\\
		&\leq \le\frac{1}{2} - \frac{1}{2\gg} \pr\le \int_{\R^N}|\Delta u|^2\d x\pr^\frac{N}{4}\underbrace{\le 2\gg\int_{\R^N}G(u)\d x \pr ^{-\frac{N-4}{4}}}_{=\le \int_{\R^N}|\Delta u |^2\d x \pr ^{-\frac{N-4}{4}}}\\
		&= \le\frac{1}{2} - \frac{1}{2\gg} \pr \int_{\R^N}|\Delta u|^2\d x = J(u).
	\end{aligned}
	\]
	Thus $J_{\e_n} ( u_{\e_n})  \leq \inf_\cM J$ and 
	\begin{equation}\label{eq:1.38}
		\int_{\R^N}|\Delta u_{\e_n}|^2 \d x \leq \le \frac{1}{2}- \frac{1}{2\gg}\pr^{-1} \inf_\cM J,\qquad\text{for every }n.
	\end{equation}
	We have $G_\e(s) \leq G_{1/2}(s)$, for all $s\in \R$ and $\e\in (0,1/2)$, so 
	\[
	\int_{\R^N} G_{1/2}(u_\e)\d x \geq \int_{\R^N} G_\e(u_\e)\d x = \frac{1}{2\gg}\int_{\R^N} |\Delta u_\e|^2\d x> 0\implies u_\e \in \cP_{1/2},
	\]
	and some calculations yield
	\[
	J_\e(u_\e) \geq J_{1/2}(m_{\cP_{1/2}}(u_\e))\geq J_{1/2}(u_{1/2}).
	\]
	Therefore, we get
	\[
	2\gg \int_{\R^N}G_+(u_{\e_n})\d x \geq \int_{\R^N}|\Delta u_{\e_n}|^2\d x  = \le \frac{1}{2}-\frac{1}{2\gg}\pr ^{-1}J_{\e_n}(u_{\e_n})\geq \le \frac{1}{2}-\frac{1}{2\gg}\pr ^{-1} J_{1/2}(u_{1/2}) >0.
	\]
	By \eqref{eq:1.38}, $(u_{\e_n})$ is bounded in $\D$ and $\int_{\R^N} G_+(u_{\e_n})\d x>c>0$, for some constant $c$.
	In view of Lemma \ref{lem:Lions}, \eqref{eq:LionsCond11} is not satisfied.
	Passing to a subsequence, there is $(y_n)$ in $\R^N$ such that $u_{\e_n}(\cdot + y_n)\rightharpoonup u_0\neq 0$ and $u_{\e_n}(x+y_n) \to u_0(x)$ a.e. in $\R^N$.
	We write $\tu_n:= u_{\e_n}(\cdot+y_n)$ for short.
	Since $g_-$ is continuous and $g_-(0)=0$, one may check that, for every $v\in C^\infty_0(\R^N)$,
	\[
	\left | \frac{1}{\e_n^{2\gg-1}}|\tu_n|^{2\gg-1}\chi_{\left\{|\tu_n|\leq \e_n\right\}} g_-(\tu_n)v \right|\leq \left|\chi_{\left\{|\tu_n|\leq \e_n\right\}} g_-(\tu_n)v  \right|  \to 0 \quad\text{a.e. in }\R^N
	\] 
	and 
	\[
	\left | \chi_{\left\{|\tu_n|> \e_n\right\}} g_-(\tu_n)v - g_-(u_0)v  \right| \to 0 \quad\text{a.e. in }\R^N.
	\]
	Due to the estimate $|g_-(\tu_n)v|\leq c\le 1 + |\tu_n|^{2\gg-1}\pr |v|$, the family $\{g_-(\tu_n)v\}$ is uniformly integrable (and tight because of the compact support).
	In view of Vitali's convergence theorem
	\[
	\begin{aligned}
		&\int_{\R^N} \left | \f_{\e_n}(\tu_n)g_-(\tu_n)v - g_-(u_0)v\right | \d x \\
		&\quad \leq \int_{\R^N}\left | \frac{1}{\e_n^{2\gg-1}}|\tu_n|^{2\gg-1}\chi_{\left\{|\tu_n|\leq \e_n\right\}} g_-(\tu_n)v \right|\d x + \int_{\R^N}\left | \chi_{\left\{|\tu_n|> \e_n\right\}} g_-(\tu_n)v - g_-(u_0)v  \right|\d x\to 0,
	\end{aligned}
	\]
	as $n\to \infty$.
	Similarly, we obtain
	\[
	\int_{\R^N} g_+(\tu_n)v \d x \to \int_{\R^N} g_+(u_0)v \d x.
	\]
	Gathering the above we deduce that
	\[
	\begin{aligned}
		J_{\e_n}^\prime(\tu_n)(v) &= \int_{\R^N}\Delta \tu_n \Delta v \d x - \int_{\R^N}g_+(\tu_n)v \d x + \int_{\R^N}\f_{\e_n}(\tu_n) g_-(\tu_n) v\d x\\
		&\to \int_{\R^N}\Delta u_0 \Delta v \d x - \int_{\R^N}g_+(u_0)v\d x + \int_{\R^N} g_-(u_0)v\d x.
	\end{aligned}
	\]
	Each $\tu_n$ is  a critical point of $J_{\e_n}$, since so is $u_{\e_n}$ (translation invariance), hence
	\[
	\int_{\R^N}\Delta u_0 \Delta v \d x =  \int_{\R^N}g(u_0)v\d x,
	\]
	i.e., $u_0$ is a weak solution to \eqref{eq}.
	By Lebesgue's dominated convergence theorem one may show that
	\[
	G^{\e_n}_-(\tu_n) \to G_-(u_0) \quad \text{a.e. in }\R^N,
	\]
	as $n\to\infty$, and, on the other hand, 
	\[
	2\gg \int_{\R^N}G^{\e_n}_-(\tu_n)\d x = 2\gg \int_{\R^N}G_+(\tu_n)\d x - \int_{\R^N}|\Delta \tu_n|^2\d x\leq c\le \sup_{n\geq 1}\|\tu_n\|_{\D}\pr <\infty,
	\]
	where we used the fact that $\tu_n\in \cM_{\e_n}$, \eqref{eq:1.39} and \eqref{eq:1.38}.
	By Fatou's lemma and by the above
	\[
	\int_{\R^N}G_-(u_0)\d x\leq \liminf_{n\to\infty}\int_{\R^N}G^{\e_n}_-(\tu_n)\d x < \infty,
	\]
	namely, we have shown that $G_-(u_0)\in L^1(\R^N)$.
	By the Poho\v{z}aev identity, we infer that $u_0\in \cM$.
	Lastly, we show that $J(u_0) = \inf_\cM J$.
	We use the weak l.s.c. of the norm and \eqref{eq:1.38} to find that
	\[
	\begin{aligned}
		J(u_0) &= \le \frac{1}{2}- \frac{1}{2\gg}\pr\int_{\R^N}|\Delta u_0|^2\d x\leq \liminf_{n\to\infty} \le \frac{1}{2}- \frac{1}{2\gg}\pr\int_{\R^N}|\Delta \tu_n|^2\d x\\
		&= \liminf_{n\to\infty} \le \frac{1}{2}- \frac{1}{2\gg}\pr\int_{\R^N}|\Delta u_{\e_n}|^2\d x\leq \inf_\cM J.
	\end{aligned}
	\]
\end{proof}

\section{Biharmonic logarithmic inequality}\label{sec:BihLog}

\begin{lemma}\label{lem:ineq}
	If $u\in \D$ and $\int_{\R^N} |u|^2 \d x = 1$, then
	\[
	\int_{\R^N}|\nabla u|^2 \d x<\le \int_{\R^N} |\Delta u |^2 \d x\pr^{1/2}.
	\]
\end{lemma}
\begin{proof}
We rely on ideas from \cite{Bellazzini}.
Let us define the Fourier transform $\widehat{u}$ of $u$ (whenever possible) as
\[
\widehat{u}(\xi) = \frac{1}{\le 2\pi\pr ^{N/2}}\int_{\R^N} e^{-ix\cdot \xi} u(x)\d x,\quad \xi \in \R^N.
\]
If $u\in \D$ and $\int_{\R^N} |u|^2 \d x = 1$, then $u\in H^2(\R^N)$ and by the Plancharel theorem
\[
\begin{aligned}
\|u\|_{L^2(\R^N)} &= \|\widehat{u}\|_{L^2(\R^N)}  ,\\
\|\nabla u\|_{L^2(\R^N)}  &= \|\widehat{\nabla u}\|_{L^2(\R^N)} =  \|\xi\widehat{u}\|_{L^2(\R^N)},\\
\|\Delta u \|_{L^2(\R^N)} &= \|\widehat{\Delta u }\|_{L^2(\R^N)} = \||\xi|^2\widehat{u}\|_{L^2(\R^N)}.
\end{aligned}
\]
By the Cauchy--Schwartz inequality we get
\[
\le\int_{\R^N}| \xi\widehat{u}(\xi)|^2\,d\xi\pr^{1/2}\leq 
	\le\int_{\R^N}| |\xi|^2\widehat{u}(\xi)|^2\,d\xi\pr^{1/4}\le\int_{\R^N}| \widehat{u}(\xi)|^2\,d\xi\pr^{1/4}.
\]
and the assertion follows with the non-strict inequality.
Recall that the equality in the Cauchy--Schwartz inequality holds if and only if $|\xi|^2\widehat{u}(\xi)=\lambda\widehat{u}(\xi)$ for some $\lambda$, what implies $\widehat{u}=0$. Hence the inequality in the statement is in fact strict.
\end{proof}

\begin{altproof}{Theorem \ref{th:BihLog}}
Observe that the following inequality holds
\begin{equation}\label{eq:ineq}
	\Big(\int_{\R^N}|\Delta u|^2\, dx\Big)^{\frac{N}{N-4}}\geq C_{N,log}\int_{\R^N}|u|^2\log |u|\, dx,\quad\hbox{for any }u\in\D,
\end{equation}
where 
$$C_{N,log}=2\gg\Big(\frac12-\frac{1}{2\gg}\Big)^{-\frac{4}{N-4}}(\inf_{\cM} J)^{\frac{4}{N-4}}
.
$$
Indeed, it is enough to consider $u\in \D$ such that $\int_{\R^N}|u|^2\log |u|\, dx>0$.
We then obtain $u(r\cdot)\in \cM$, where 
$$r:=\le\frac{ 2\gg\int_{\R^N}|u|^2\log |u|\d x}{\int_{\R^N}|\Delta u|^2}\pr ^{1/4}.$$
Hence $J(u(r\cdot))\geq \inf_{\cM} J$ and we get \eqref{eq:ineq}.

Now note that
\eqref{eq:ineq}  is equivalent to
\begin{equation}\label{eq:scaled}
	\Big(\int_{\R^N}|\Delta u|^2\,dx\Big)^{\frac{N}{N-4}}\geq C_{N,log}\max_{\alpha\in\R}\Big\{ e^{-\alpha2\gg}\int_{\R^N}|e^{\alpha}u|^2\log |e^{\alpha}u|\,dx\Big\},\quad\hbox{ for } u\in\cD^{2,2}(\R^N).
\end{equation}
Assuming that $\int_{\R^N}|u|^2\, dx=1$,
the maximum of the right hand side of \eqref{eq:scaled} is attained at $\alpha=\frac{N-4}{8}-\int_{\R^N}|u|^2\log |u|\,dx$. Hence we get
\begin{eqnarray*}
	\frac{N}{N-4}\log \Big(\int_{\R^N}|\Delta u|^2\,dx\Big)\geq \log(C_{N,log})-\alpha 2\gg+2\alpha+\log\Big(\frac{N-4}{8}\Big)
\end{eqnarray*}
that is
\begin{eqnarray*}
	\frac{N}{N-4}\log \Big(\int_{\R^N}|\Delta u|^2\,dx\Big)\geq \log\Big(C_{N,log}\frac{N-4}{8}e^{-1}\Big)+\frac{8}{N-4}\int_{\R^N}|u|^2\log |u|\,dx
\end{eqnarray*}
and
\begin{eqnarray*}
	\frac{N}{8}\log \Big(\int_{\R^N}|\Delta u|^2\,dx\Big)\geq \frac{N-4}{8} \log\Big(C_{N,log}\frac{N-4}{8}e^{-1}\Big)+\int_{\R^N}|u|^2\log |u|\,dx
\end{eqnarray*}
thus \eqref{logSob} holds.

We show that the constant in \eqref{logSob} is optimal, i.e., there is $u\in \D$ such that the equality holds.
First of all, notice that if $u_0$ is a minimizer given by Theorem \ref{thm:2}, then for $u_0$ we have the equality in \eqref{eq:ineq}:
\begin{equation}\label{eq:6.1}
\le\int_{\R^N}|\Delta u_0|^2\, dx\pr^{\frac{N}{N-4}}= C_{N,log}\int_{\R^N}|u_0|^2\log |u_0|\, dx.
\end{equation}
We use \eqref{eq:ineq} for the family of functions $\frac{e^\alpha}{\|u_0\|_{L^2}}u_0\in \D$, $\alpha\in \R$, to get 
\begin{equation}\label{ineq:1}
\le\int_{\R^N}|\Delta u_0|^2\, dx\pr^{\frac{N}{N-4}}\geq  C_{N,log}\|u_0\|_{L^2}^{2\gg-2}e^{(2-2\gg)\alpha}\int_{\R^N}|u_0|^2\log \left |\frac{e^{\alpha}}{\|u_0\|_{L^2}} u_0\right|\,dx,\quad \alpha\in \R.
\end{equation}
Now let us consider the function $f:\R\to\R$ given by
\[
\begin{aligned}
f(\alpha)&:= C_{N,log}\|u_0\|_{L^2}^{2\gg-2}e^{(2-2\gg)\alpha}\int_{\R^N}|u_0|^2\log \left |\frac{e^{\alpha}}{\|u_0\|_{L^2}} u_0\right|\,dx - \le\int_{\R^N}|\Delta u_0|^2\, dx\pr^{\frac{N}{N-4}}
\end{aligned}
\]
Note that
\[
f^\prime(\alpha) = 0 \iff
\alpha = \frac{N-4}{8} - \int_{\R^N} \left |\frac{u_0}{\|u_0\|_{L^2}}\right|^2\log\left |\frac{u_0}{\|u_0\|_{L^2}}\right| \d x.
\]
On the other hand, $f$ attains maximum at $\alpha = \log(\|u_0\|_{L^2})$ in view of \eqref{ineq:1} and \eqref{eq:6.1}, thus
\[
\int_{\R^N} \left |\frac{u_0}{\|u_0\|_{L^2}}\right|^2\log\left |\frac{u_0}{\|u_0\|_{L^2}}\right| \d x = \frac{N-4}{8} - \log(\|u_0\|_{L^2})
\]
or, equivalently, 
\[
\frac{1}{\|u_0\|_{L^2}^2}\int_{\R^N} |\Delta u_0 |^2\d x = \frac{N}{4},
\]
where we used the fact that $u_0\in \cM$.
Therefore we obtain the equality in \eqref{logSob} for the function $\frac{u_0}{\|u_0\|_{L^2}}$.

Let us now suppose that
$$\frac{N}{8}\log \le \le \frac{8e}{C_{N,log}(N-4)}\pr^{(N-4)/N}\int_{\R^N}|\Delta u|^2\,dx \pr=\int_{\R^N}|u|^2\log |u|\,dx$$
for some $u\in\D$ such that $\|u\|_{L^2(\R^N)}=1$. Then
$$\Big(\int_{\R^N}|\Delta u|^2\,dx\Big)^{\frac{N}{N-4}}=C_{N,log} e^{-\alpha2\gg}\int_{\R^N}|e^{\alpha}u|^2\log |e^{\alpha}u|\,dx$$
for $\alpha=\frac{N-4}{8}-\int_{\R^N}|u|^2\log |u|\,dx$ and the equality in \eqref{eq:ineq} holds for $u_1:=e^{\alpha}u$. Hence $J(u_0)=\inf_{\cM}J$ for 
$$u_0:=u_1(r\cdot)\in \cM, \quad \text{where } r = \le \frac{2\gg\int_{\R^N}|u_1|^2\log|u_1|\d x}{\int_{\R^N} |\Delta u_1|^2\d x}\pr^{1/4}. $$
Let us sketch the proof that $u_0$ is a critical point of $J$.
Firstly, note that, for every $v\in C^\infty_0(\R^N)$, $G(u_0 +v)\in L^1(\R^N)$, for $G(s) := s^2\log|s|$.
Fix an arbitrary $v\in C^\infty_0(\R^N)$.
We use the fact that $G$ is $C^1$-smooth  and the Lebesgue dominated convergence theorem to get
\[
\lim_{t\to 0} \frac{1}{t}\le\int_{\R^N} G(u_0+tv)\d x - \int_{\R^N} G(u_0)\d x\pr = \int_{\R^N} g(u_0)v\d x.
\] 
By the continuity, $\int_{\R^N}G(u_0 +tv)\d x >0$, for sufficiently small $|t|>0$, so $(u_0+tv)(r\cdot)\in \cM$, where
\[
r =  \le \frac{2\gg\int_{\R^N}G(u_0+tv)\d x}{\int_{\R^N}|\Delta(u_0 +tv)|^2\d x}\pr ^{1/4}.
\]
Hence
\[
J((u_0+tv)(r\cdot))\geq \inf_{\cM}J = J(u_0)
\]
or, equivalently,
\[
\le \frac{1}{2}-\frac{1}{2\gg}\pr^{4/N} \int_{\R^N} |\Delta( u_0+tv)|^2\d x \geq J(u_0)^{4/N} \le 2\gg \int_{\R^N}G(u_0+tv)\d x\pr^{(N-4)/N}.
\]
We then proceed similarly as in the last part of the proof of Lemma \ref{lem:theta} to conclude that
\[
\int_{\R^N} \Delta u_0 \Delta v \d x \geq \int_{\R^N} g(u_0) v \d x,
\]
which yields that $u_0$ is a critical point of $J$.

Finally, we show the estimate of the constant $C_{N,log}$ from Theorem \ref{logSob}.
Observe that if  $u\in \D$ and $\int_{\R^N} |u|^2 \d x = 1$, then $u\in H^2(\R^N)$.
In view of Lemma \ref{lem:ineq} and the logarithmic Sobolev inequality \eqref{eq:ineqLogSob} we obtain
\[
\int_{\R^N} |u|^2\log(|u|)\d x < \frac{N}{4}\log \le \frac{2}{\pi e N}\le \int_{\R^N} |\Delta u |^2\d x\pr ^{1/2} \pr =  \frac{N}{8}\log \le \Big(\frac{2}{\pi e N}\Big)^2 \int_{\R^N} |\Delta u |^2\d x \pr,
\]
and so
$$\le \frac{8e}{C_{N,log}(N-4)}\pr^{(N-4)/N} < \Big(\frac{2}{\pi e N}\Big)^2.$$

\end{altproof}

\section*{Acknowledgements}

The authors were supported by the National Science Centre, Poland (Grant No. 2017/26/E/ST1/00817).
J. Mederski was also partially supported by the Deutsche Forschungsgemeinschaft (DFG, German Research Foundation) – Project-ID 258734477 – SFB 1173 during the stay at Karlsruhe Institute of Technology.

\end{document}